\documentclass[11pt]{article}

\usepackage[left=1.2in,right=1.2in]{geometry}

\usepackage[utf8]{inputenc}
\usepackage[T1]{fontenc}
\usepackage{ifthen}
\usepackage{amssymb}
\usepackage{amsthm} 
\usepackage{bbm}
\usepackage{enumitem}

\usepackage{amsmath}
\usepackage{blkarray}

\usepackage[all]{xy}

\usepackage{xcolor}
\definecolor{notecolor}{rgb}{1,0,0}

\usepackage{tikz}
\usepackage{tkz-graph}
\usepackage{tikz-cd}
\usetikzlibrary{arrows.meta, positioning}
\usepackage{xparse}
\usepackage{mathrsfs}
\usepackage{graphicx}
\usepackage{hyperref}

\definecolor{linkcolor}{rgb}{0,0,0.8} 

{
\theoremstyle{plain}
\newtheorem{theorem}{Theorem}[section]
\newtheorem{lemma}[theorem]{Lemma}
\newtheorem{proposition}[theorem]{Proposition}
\newtheorem{corollary}[theorem]{Corollary}
\newtheorem{example}[theorem]{Example}
\newtheorem{definition}[theorem]{Definition}
\newtheorem{conjecture}[theorem]{Conjecture}
\newtheorem*{conjecture*}{Conjecture}

}

{
\theoremstyle{definition}

\newtheorem{remark}[theorem]{Remark}
}

\newcommand{\sizedescriptor}[2]
{
\ifthenelse{\equal{#1}{0}}{}{
\ifthenelse{\equal{#1}{1}}{\big}{
\ifthenelse{\equal{#1}{2}}{\Big}{
\ifthenelse{\equal{#1}{3}}{\bigg}{
\ifthenelse{\equal{#1}{4}}{\Bigg}{
#2}}}}}
}

\newcommand{\proven}[1]{\underline{#1} \\ \ \vspace{-2ex} \\}

\newcommand{\all}[1]{\forall #1 .\,}  

\newcommand{\df}[1]{\emph{\textbf{#1}}}  

\NewDocumentCommand{\set}
	{O{auto} m G{\empty}}
	{\sizedescriptor{#1}{\left}\{ {#2} \ifthenelse{\equal{#3}{}}{}{ \; \sizedescriptor{#1}{\middle}| \; {#3}} \sizedescriptor{#1}{\right}\}}

\newcommand{\pst}{\mathcal{P}}  

\newcommand{\NN}{\mathbb{N}}

\newcommand{\RR}{\mathbb{R}}

\newcommand{\intoo}[3][\RR]{{#1}_{(#2, #3)}}
\newcommand{\intcc}[3][\RR]{{#1}_{[#2, #3]}}
\newcommand{\intoc}[3][\RR]{{#1}_{(#2, #3]}}
\newcommand{\intco}[3][\RR]{{#1}_{[#2, #3)}}

\newcommand{\er}{\overline{\RR}}

\NewDocumentCommand{\oball}  
	{O{\empty} G{\empty} G{\empty}}
	{B_{#1}\ifthenelse{\equal{#2}{}}{}{\!\left(#2, #3\right)}}
\NewDocumentCommand{\cball}  
	{O{\empty} G{\empty} G{\empty}}
	{\overline{B}_{#1}\ifthenelse{\equal{#2}{}}{}{\!\left(#2, #3\right)}}
\NewDocumentCommand{\cth}  
	{O{\empty} G{\empty} G{\empty}}
	{\overline{\mathrm{th}}_{#1}\ifthenelse{\equal{#2}{}}{}{\!\left(#2, #3\right)}}

\newcommand{\rstr}[1]{\left.{#1}\right|}  
\newcommand{\parto}{\mathrel{\rightharpoonup}}  

\NewDocumentCommand{\dimg}  
	{O{\empty} m G{\empty}}
	{{#2}_*\ifthenelse{\equal{#3}{}}{}{\!\sizedescriptor{#1}{\left}( {#3} \sizedescriptor{#1}{\right})}}
\NewDocumentCommand{\pimg}  
	{O{\empty} m G{\empty}}
	{{#2}^*\ifthenelse{\equal{#3}{}}{}{\!\sizedescriptor{#1}{\left}( {#3} \sizedescriptor{#1}{\right})}}

\newcommand{\transpose}{{\mathsf{T}}}

\usepackage{mathtools}
\mathtoolsset{centercolon}

\newcommand{\mg}[1][]{\mathsf{mag}_{#1}}

\newcommand{\fsub}{\mathrm{Fin}}
\newcommand{\ifsub}{\fsub_{+}}
\newcommand{\csub}{\mathrm{Cmp}}
\newcommand{\icsub}{\csub_{+}}

\newcommand{\adj}{\mathrm{adj}}

\newcommand{\cl}[1][]{\mathrm{Cl}_{#1}}

\title{Tractable Metric Spaces and Magnitude Continuity} 

\author{
Sara Kali\v{s}nik\thanks{
 Pennsylvania State University,
\texttt{skalisnik@psu.edu}}\phantom{x} and 
Davorin Le\v{s}nik\thanks{
University of Ljubljana,
\texttt{davorin.lesnik@fmf.uni-lj.si}}
}

\date{}

\begin{document}
\maketitle

\begin{abstract}
Magnitude is an isometric invariant of metric spaces introduced by Leinster. Since its inception, it has inspired active research into its connections with integral geometry, geometric measure theory, fractal dimensions, persistent homology, and applications in machine learning. In particular, when it comes to applications, continuity and stability of invariants play an important role. Although it has been shown that magnitude is nowhere continuous on the Gromov–Hausdorff space of finite metric spaces, positive results are possible if we restrict the ambient space. In this paper, we introduce the notion of \emph{tractable metric spaces}, provide a characterization of these spaces, and establish several continuity results for magnitude in this setting. As a consequence, we offer a new proof of a known result stating that magnitude is continuous on the space of compact subsets of~\(\mathbb{R}\) with respect to the Hausdorff metric. Furthermore, we show that magnitude is Lipschitz continuous when restricted to bounded subspaces of~\(\mathbb{R}\).
\end{abstract}


\section{Introduction}
Magnitude is a cardinality-like invariant of finite metric spaces introduced by Leinster in the 2010s~\cite{leinster2010}. It originated in category theory as a generalization of the notion of Euler characteristic to enriched categories. Magnitude was later extended to compact metric spaces in multiple ways in~\cite{leinster2010,LW13,Willerton2014,W09,meckes2013}.


Magnitude has been shown to encode many invariants from integral geometry and geometric measure theory, including volume, capacity, dimension, and intrinsic volumes~\cite{leinster2010,Meckes2015,Willerton2014,LW13}.  It has also been applied in other contexts --- for example, to quantify biodiversity~\cite{Solow1994,LC12}, in machine learning, to measure the intrinsic diversity of latent representations~\cite{magnitude-ML} and in connection with persistent homology~\cite{O18,OMALLEY2023107396,GH21}. While the question of continuity of magnitude is mathematically interesting in its own right~\cite{LM17,meckes2013,so2025maximumdiversityweightinginvariants,so2026convergencemagnitudefinitepositive}, it becomes especially crucial in the context of data analysis~\cite{katsumasa2025magnitudegenericallycontinuousfinite}, where a fundamental principle is that a good invariant should depend continuously on the input data.

Unfortunately, if we consider the set of isometry classes of finite metric spaces, equipped with the Gromov-Hausdorff metric, magnitude is not continuous~\cite{leinster2010,Roff25}. In fact, the authors in~\cite{katsumasa2025magnitudegenericallycontinuousfinite} show that magnitude is nowhere continuous on the Gromov-Hausdorff space of finite metric spaces. However, as they point out, continuity results are feasible if we restrict the ambient space.

To this end, we introduce the class of \emph{tractable metric spaces}. These are positive definite, Heine--Borel metric spaces in which all closed balls have finite magnitude. Examples include \( \ell_p^N \) for \( 1 \leq p \leq 2 \), finite-dimensional vector subspaces of~\(L^1\), and their closed subspaces. We demonstrate that tractable metric spaces have properties that are naturally connected with the question of continuity of magnitude. For example, in a locally compact metric space, if the magnitude is finite and continuous at some compact subspace, that subspace is necessarily contained in some tractable neighborhood. Among others, we also establish the following equivalence: the continuity of magnitude on the space of compact subspaces of a tractable metric space~$M$ is equivalent to the uniform continuity of magnitude on the space of finite subspaces of every bounded subspace \( B \subseteq M \), and, in turn, to uniform continuity on the space of finite subspaces within each closed ball \( \cball[M]{x}{r} \subseteq M \). We also explore the one-point property~\cite{Roff25,leinster2023spaces} in the context of tractable spaces. In particular, in a tractable star-shaped subset of a normed vector space, magnitude of compact/finite subsets is continuous at the star center if and only if the magnitude function of every non-empty compact/finite subset has the one-point property. We also show that in any tractable normed vector space, magnitude is continuous at every singleton.

These results build a framework for deriving continuity properties of magnitude in concrete settings. For example, we can deduce that magnitude is continuous on the class of compact metric subspaces of $\RR$ equipped with the Hausdorff metric (this statement has already appeared in the literature, see~\cite[Corollary 5.3.13]{LM17} and \cite[Corollary 2.12]{meckes2013}). We also prove that magnitude is Lipschitz continuous when restricted to compact subsets of bounded subspaces of~\(\mathbb{R}\) (Theorem~\ref{thm:continuity_on_R}).


The paper is organized as follows: in Subsection~\ref{sub:notation} we introduce the basic notation and conventions used throughout the paper. In Section~\ref{section:prelims} we review the basic definitions from the theory of magnitude: we begin by defining the magnitude of matrices, then extend the notion to finite metric spaces, and finally to compact metric spaces. 
Section~\ref{section:tractable_spaces} is devoted to tractable metric spaces. We begin by defining this class of spaces and providing examples. We then study their properties and explore their connection to the continuity of magnitude. This section lays the foundation for the continuity results established in Section~\ref{section:continuity_of_magnitude}.
Finally, in Section~\ref{section:continuity_of_magnitude}, we study the continuity properties of magnitude on finite and compact subspaces of the real line. We show that magnitude is continuous on the space of compact subspaces of~$\RR$, equipped with the Hausdorff metric. We further prove that magnitude is Lipschitz continuous when restricted to bounded subspaces of~$\RR$. We conclude the paper with Section~\ref{section:contribution}, where we summarize the results and list some open questions pertaining to the continuity of magnitude in the context of tractable metric spaces.

\subsection{Notation}\label{sub:notation}

We denote the set of natural numbers by~$\NN$. We treat~$0$ as a natural number, so $\NN = \set{0, 1, 2, 3,\ldots}$.

We denote the set of real numbers by~$\RR$, and the set of extended real numbers (with infinities included) by~$\er$. That is, $\er = \RR \cup \set{-\infty, \infty}$.

Subsets, given by a relation, are denoted by that relation in the index; for example, $\RR_{\geq 0}$ is the set of non-negative real numbers. Intervals between two numbers are denoted by these two numbers in brackets and in the index. Round, or open, brackets $(\ )$ denote the absence of the boundary in the set, and square, or closed, brackets $[\ ]$ its presence. For example, $\intcc{0}{1} = \set{x \in \RR}{0 \leq x \leq 1}$ is the usual closed unit interval, and $\intco[\NN]{5}{10} = \set{n \in \NN}{5 \leq n < 10} = \set{5, 6, 7, 8, 9}$.

The powerset of a set~$A$ (the set of all subsets of~$A$) is denoted by~$\pst(A)$. We denote the set of finite subsets of~$A$ by~$\fsub(A)$, and the set of non-empty finite subsets by~$\ifsub(A)$. If $A$ is also equipped with a topology, then $\csub(A)$ denotes the set of compact subspaces of~$A$, and $\icsub(A)$ the set of non-empty compact subspaces of~$A$.

A function $f$ mapping from a set~$A$ to a set~$B$ is denoted as $f\colon A \to B$. If $f$ is merely a partial map (not necessarily defined on the whole~$A$), we write this as $f\colon A \parto B$.


Following the notation from~\cite{leinster2023spaces}, for any $p \in \er_{\geq 1}$ and $N \in \NN$, let $\ell_p^N$ denote the Banach space~$\RR^N$, equipped with the $p$-norm (and the induced $p$-metric), and let $\ell_p$ denote the infinite-dimensional version of this space. Moreover, $L_1$ is the shorthand for $L_1\big(\intcc{0}{1}, \RR\big)$, i.e.\ the space of (equivalence classes of) measurable functions $\intcc{0}{1} \to \RR$, equipped with the usual integral $1$-norm.

Given a metric space~$M$ with a metric~$d$,
\begin{itemize}
\item
the open ball in~$M$ with the center in $x \in M$ and radius $r$ is denoted by $\oball[M]{x}{r}$, likewise for the closed ball $\cball[M]{x}{r}$;
\item
the closure of a subset $A \subseteq M$ in~$M$ is denoted by~$\cl[M](A)$;
\item
the (closed) $r$-thickening of~$A$ in~$M$ for a subset $A \subseteq M$ and $r \in \RR_{\geq 0}$ is denoted as
\[\cth[M](A, r) := \set{x \in X}{d(x, A) \leq r}.\]
\end{itemize}

\section{Preliminaries}\label{section:prelims}

First, we review some basic definitions and standard results concerning the magnitude of matrices, as well as the magnitude of finite and compact subsets of metric spaces~\cite{leinster2010,meckes2013}.

\subsection{The Magnitude of a Matrix} 

The magnitude of a matrix serves as the foundation for the notion of magnitude in metric spaces. This concept, originally introduced in the context of enriched category theory~\cite{leinster2010}, assigns a real number to certain square matrices via the notion of weightings.

\begin{definition} 
A vector $w \in \RR^n$ is called a \df{weighting} for a matrix $A \in \RR^{n \times n}$ when $A w = 1_n$, where $1_n \in \RR^n$ is the vector with all components equal to~$1$. A \df{coweighting} is a weighting of the transposed matrix~$A^\transpose$.
\end{definition}

A matrix may or may not have a weighting, and independently may or may not have a coweighting (for example, $\begin{psmallmatrix} 1 & 0 \\ 1 & 0 \end{psmallmatrix}$ has infinitely many weightings, but no coweightings). However, a standard argument shows that if a matrix~$A$ has at least one weighting and at least one coweighting, then the sum of components is the same for all weightings and all coweightings~\cite[Lemma~1.1.2]{leinster2010}. This sum is called the magnitude of~$A$.

\begin{definition}
Suppose there exists a weighting~$w$ and a coweighting~$v$ for a matrix $A \in \RR^{n \times n}$. The \df{magnitude} of matrix~$A$ is defined to be
\[\mg(A) := \sum_{i \in \intcc[\NN]{1}{n}} w_i,\]
where $w_i$ is the $i$-th component of the weighting~$w$. This gives us, for each $n \in \NN$, a partial map $\mg \colon \RR^{n \times n} \parto \RR$.
\end{definition}

We make the following remarks with regard to this definition.

\begin{itemize}
\item
The magnitude of a matrix is typically denoted by~$|A|$. To avoid confusion with absolute value notation, we instead write $\mg(A)$ for the magnitude of~$A$.
\item
If a matrix is symmetric, then it has a weighting if and only if it has a coweighting. Thus, for symmetric matrices, one needs only to require a weighting for the magnitude to be defined. In this paper, all matrices in question will be symmetric, as they will be similarity matrices of metric spaces. However, if one studies more general Lawvere metric spaces, having magnitude for non-symmetric matrices is important.
\item
If $A \in \RR^{n \times n}$ is invertible, then it has a unique weighting $A^{-1} 1_n$, and a unique coweighting $(A^{-1})^\transpose 1_n$. It follows that
\[\mg(A) = \frac{\mathrm{sum}(\adj(A))}{\det(A)}.\]
Thus the partial map $\mg\colon \RR^{n \times n} \parto \RR$ restricts to a \emph{total} map $\mg \colon GL_n(\RR) \to \RR$. It is clear from the formula that this restriction is continuous.
\item
If $A \in \RR^{n \times n}$ is positive definite, then it is symmetric and invertible. 
\item
The case $n = 0$ is allowed; the unique $(0 \times 0)$-matrix (the ``empty matrix'') is positive definite, its determinant is~$1$, its adjugate and inverse are again the empty matrix, and its magnitude is a sum of zero summands, so~$0$.
\end{itemize}

We can see that while the magnitude is only partially defined on the space of all square matrices, it is a well-behaved and continuous function on several important subclasses, including invertible and positive definite matrices. 

\subsection{The Magnitude of Finite Metric Spaces}\label{subsection:finite-metric-space-magnitude}

In this subsection, we recall the definition of magnitude of a finite metric space~\cite{leinster2010}.

\begin{definition}\label{definition:finite-metric-space-magnitude}
Let $M = \set{a_1, \ldots, a_n}$ be equipped with a metric~$d$. The matrix $Z_M \in \RR^{n \times n}$, given by $(Z_M)_{i, j} := e^{-d(a_i, a_j)}$, is called the \df{similarity matrix} of~$M$.
The \df{magnitude} of the metric space~$M$, denoted by $\mg(M)$, is defined to be the magnitude of its similarity matrix.
\end{definition}

Note that for every finite metric space, its similarity matrix is symmetric. However, not every similarity matrix has a weighting, so the magnitude is not defined for every finite metric space.

In general, there are many different ways to list elements of a finite metric space: we can permute and repeat elements. Nevertheless, the magnitude turns out to be independent of this choice; in fact, it is an isometric invariant of finite metric spaces.

\begin{example}
The magnitude of the empty space is~$0$, the magnitude of a singleton is~$1$, and for two points at distance~$u \in \RR_{> 0}$, the magnitude is $\frac{2}{1 + e^{-u}} = 1 + \tanh(\frac{u}{2})$ where $\tanh$ denotes the hyperbolic tangent~\cite[Examples~2.1.1]{leinster2010}. More generally, if $a_1 < a_2 < \ldots < a_n$ are real numbers (equipped with the usual metric), then $\mg(\set{a_1, \ldots, a_n}) = 1 + \sum_{i \in \intco[\NN]{1}{n}} \tanh(\frac{a_{i+1} - a_i}{2})$~\cite[Corollary~2.3.4]{leinster2010}.
\end{example}

However, we are rarely interested in magnitude of a single metric space at a time. More commonly studied (and much more informative) is the so-called magnitude function. Given a metric space~$M$ and $t \in \RR_{> 0}$, the metric space $t\;\!M$ is defined to have the same underlying set as~$M$, but all distances are multiplied by~$t$. The \df{magnitude function} of a finite metric space~$M$ is defined to be the partial map $\RR_{> 0} \parto \RR$, $t \mapsto \mg(t\;\!M)$~\cite{leinster2010}. It turns out that for every finite metric space, its magnitude function is defined for all but finitely many values, and is continuous on its domain of definition.

The definition of magnitude readily generalizes to pseudometric space (where two different points can be at zero distance). In this sense, the magnitude function can be extended to $t = 0$. It is easy to calculate that for any non-empty finite metric space~$M$ we have $\mg(0\;\!M) = 1$. This extended magnitude function is no longer necessarily continuous, however; in fact, the discontinuity can be arbitrarily bad~\cite[Theorem~3.8]{Roff25}. A non-empty finite metric space~$M$ is said to have the \df{one-point property} (since $0\;\!M$ looks like one point) when its magnitude function is continuous at $t = 0$, i.e.~when $\lim_{t \to 0} \mg(t\;\!M) = 1$. Which spaces have the one-point property is still a topic of ongoing research.

In this paper we are interested in a more general form of magnitude continuity. For any metric space~$M$, we can study the magnitude of all its finite subspaces. This gives us a partial function $\mg[\fsub(M) ]\colon \fsub(M) \parto \RR$ (where $\fsub(M)$ denotes the set of all finite subsets of~$M$). This could rightfully be also called ``magnitude function'', but to avoid any confusion with the aforementioned notion, we just refer to it as ``magnitude''. 


For any metric space $M$ with a metric~$d$, we consider $\fsub(M)$ to be equipped with the Hausdorff metric~$d_H$. This is only an extended metric, however, since the empty subset is at infinite distance from the others. For convenience, we will usually restrict ourselves to~$\ifsub(M)$, i.e.\ the set of all \emph{non-empty} subsets of~$M$, where the Hausdorff metric is indeed a metric. This does not change any discussion on continuity: since $\emptyset$ is an isolated point, any (partial) map with domain $\fsub(M)$ is continuous if and only if its restriction to~$\ifsub(M)$ is.


Magnitude for finite subspaces can be seen as a generalization of the magnitude function (aside from the trivial case of the empty space whose magnitude function is constantly~$0$), as we now demonstrate.

\begin{proposition}
Let $M$ be a finite metric space with metric~$d$ and let $a \in M$. Then
\[\rho\big((x, t), (y, u)\big) := \begin{cases} t \cdot d(x, y) & \text{if $t = u$}, \\ t \cdot d(x, a) + u \cdot d(y, a) + |t - u| & \text{if $t \neq u$} \end{cases}\]
is a metric on $M \times \RR_{> 0}$. Moreover, for this metric, for every $t \in \RR_{> 0}$ the map $t\;\!M \to M \times \RR_{> 0}$, $x \mapsto (x, t)$ is an isometric embedding.
\end{proposition}

\begin{proof}
We clearly have $\rho\big((x, t), (x, t)\big) = 0$. Conversely, assume $\rho\big((x, t), (y, u)\big) = 0$. If we had $t \neq u$, then $\rho\big((x, t), (y, u)\big) = t \cdot d(x, a) + u \cdot d(y, a) + |t - u| \geq |t - u| > 0$, a contradiction, so $t = u$. Hence $0 = \rho\big((x, t), (y, u)\big) = t \cdot d(x, y)$, and since $t > 0$, we conclude $d(x, y) = 0$. Thus $x = y$, so $(x, t) = (y, u)$.

Obviously, $\rho$ is symmetric. It remains to verify the triangle inequality. Take any $(x, t), (y, u), (z, v) \in M \times \RR_{> 0}$.

Case~$t = u = v$:
\[\rho\big((x, t), (y, u)\big) + \rho\big((y, u), (z, v)\big) = t \cdot d(x, y) + u \cdot d(y, z) =\]
\[= t \cdot \big(d(x, y) + d(y, z)\big) \geq t \cdot d(x, z) = \rho\big((x, t), (z, v)\big).\]

Case~$t \neq u = v$:
\[\rho\big((x, t), (y, u)\big) + \rho\big((y, u), (z, v)\big) = t \cdot d(x, a) + u \cdot d(y, a) + |t - u| + u \cdot d(y, z) =\]
\[= t \cdot d(x, a) + u \cdot \big(d(y, a) + d(y, z)\big) + |t - u| \geq\]
\[\geq t \cdot d(x, a) + u \cdot d(z, a) + |t - u| = \rho\big((x, t), (z, v)\big).\]

Case~$t = u \neq v$ is treated in the same way.

Case~$t \neq u \neq v$, $t = v$:
\[\rho\big((x, t), (y, u)\big) + \rho\big((y, u), (z, v)\big) =\]
\[= t \cdot d(x, a) + u \cdot d(y, a) + |t - u| + u \cdot d(y, a) + v \cdot d(z, a) + |u - v| \geq\]
\[\geq t \cdot \big(d(x, a) + d(z, a)\big) \geq t \cdot d(x, z) = \rho\big((x, t), (z, v)\big).\]

Case~$t \neq u \neq v$, $t \neq v$:
\[\rho\big((x, t), (y, u)\big) + \rho\big((y, u), (z, v)\big) =\]
\[= t \cdot d(x, a) + u \cdot d(y, a) + |t - u| + u \cdot d(y, a) + v \cdot d(z, a) + |u - v| \geq\]
\[\geq t \cdot d(x, a) + v \cdot d(z, a) + |t - v| = \rho\big((x, t), (z, v)\big).\]

The fact that the map $t\;\!M \to M \times \RR_{> 0}$, $x \mapsto (x, t)$ is an isometric embedding is immediate from the definition of~$\rho$.
\end{proof}

It follows from this proposition that the magnitude function of~$M$ is a restriction of $\mg[\ifsub(M \times \RR_{> 0})]$, in the sense that $\mg(t\;\!M) = \mg[\ifsub(M \times \RR_{> 0})](M \times \set{t})$. If $M$ is a finite subspace of a normed vector space~$V$, we do not even need a product space, we simply have $\mg(t\;\!M) = \mg[\ifsub(V)](t \cdot M)$.

\subsection{Positive Definite Metric Spaces}

One way to guarantee that a finite metric space has magnitude is to require that it has a positive definite similarity matrix; such a metric space is also called positive definite. It follows from Sylvester's criterion that its every subspace has this property as well. Hence, the following is a generalization of this property to not necessarily finite metric spaces.

\begin{definition}
A metric space~$M$ is \df{positive definite} when its every finite subspace has a positive definite similarity matrix.
\end{definition}

\begin{example}
A standard example of a positive definite space is $L_1$~\cite[Theorem~3.6]{meckes2013}. Consequently, every metric space which isometrically embeds into~$L_1$ is positive definite as well. This includes $L_p$ and $\ell_p^N$-spaces for $p \in \intcc{1}{2}$~\cite[Theorem 2.1]{Dor1976}.   
\end{example}

Magnitude is well behaved for positive definite metric spaces~\cite{meckes2013}. To begin with, if $M$ is positive definite, then its magnitude $\mg[\fsub(M)]\colon \fsub(M) \to \RR$ is a total map. Additionally, it satisfies the following important property: for all $F', F'' \in \fsub(M)$, if $F' \subseteq F''$, then $\mg[\fsub(M)](F') \leq \mg[\fsub(M)](F'')$. We say that positive metric spaces have \df{inclusion-monotone magnitude}. For the proof of this fact, see~\cite[
Corollary 2.4.4]{leinster2010}.

\subsection{The Magnitude of Compact Metric Spaces}

Because of inclusion-monotonicity of magnitude, there is a natural extension of the definition of magnitude to compact positive definite metric spaces. The idea is that a metric compact is totally bounded, and can thus be arbitrarily well approximated with finite subsets from below.

\begin{definition}\label{definition:compact-metric-space-magnitude}
Let $K$ be a compact positive definite metric space. Then its \df{magnitude} is defined as
\[\mg(K) := \sup\set[1]{\mg[\fsub(K) ](F)}{F \in \fsub(K)}.\]
\end{definition}

In principle, there is nothing stopping us to take this as the definition of the magnitude for any metric compact, but for general (non-positive definite) spaces, this is poorly behaved. In particular, for a finite metric space which is not positive definite, Definitions~\ref{definition:finite-metric-space-magnitude} and~\ref{definition:compact-metric-space-magnitude} might not yield the same magnitude~\cite{meckes2013}.

\begin{example}
Let $a, b \in \RR$, $a \leq b$. The magnitude of the interval~$\intcc{a}{b}$~\cite[Theorem~7]{LW13} with respect to the standard Euclidean metric is given by
\[
\mg(\intcc{a}{b}) = 1 + \frac{b - a}{2}.
\]
More generally~\cite[Corollary~5.4.3]{LM17}, if \( A = \bigcup_{i \in \intcc[\NN]{1}{n}} \intcc{a_i}{b_i} \subseteq \mathbb{R} \) is a finite union of closed bounded intervals such that $a_i \leq b_i < a_{i+1} \leq b_{i+1}$ for all $i \in \intco[\NN]{1}{n}$, then the magnitude is given by
\[
\mg(A) = 1 + \sum_{i \in \intcc[\NN]{1}{n}} \frac{b_i - a_i}{2} + \sum_{i \in \intco[\NN]{1}{n}} \tanh\big(\frac{a_{i+1} - b_i}{2}\big).
\]
\end{example}

Other definitions of magnitude for compact metric spaces have appeared in the literature~\cite{leinster2010,LW13,Willerton2014,W09,meckes2013}. They have been shown to be equivalent for positive definite spaces, but not general ones~\cite{meckes2013}. In this paper, all compact spaces whose magnitude we consider are positive definite, so Definition~\ref{definition:compact-metric-space-magnitude} will work fine.

This definition allows us, for a positive definite metric space~$M$, to extend the magnitude from $\fsub(M)$ to~$\csub(M)$. However, in general, the supremum in Definition~\ref{definition:compact-metric-space-magnitude} need not be finite (see~\cite[Theorem~2.1]{leinster2023spaces}). Thus, we obtain the magnitude as a (total) map by extending the codomain to $\er$: $\mg[\csub(M)]\colon \csub(M) \to \er$. However, in this paper, we will not be interested in spaces with infinite magnitude, so for a positive definite space~$M$, the magnitude $\mg[\csub(M)]\colon \csub(M) \parto \RR$ can potentially only be a partial map.

Just as in the finite case (Subsection~\ref{subsection:finite-metric-space-magnitude}), we equip $\csub(M)$ with the Hausdorff metric. Again, this is only an extended metric, so when convenient, we restrict ourselves to~$\icsub(M)$ (the set of non-empty compact subspaces of~$M$) where the Hausdorff metric is indeed a metric. By the same argument as before, this does not affect the questions of continuity.

Leinster and Meckes~\cite[Theorem~3.1]{leinster2023spaces} prove that every compact subspace of a finite-dimensional vector subspace of~$L_1$ (in particular, every finite subspace of~$L_1$) has finite magnitude and the one-point property. Consequently, spaces that isometrically embed into a finite-dimensional vector subspace of~$L_1$, including spaces $\ell_p^N$ for $p \in \intcc{1}{2}$ and $N \in \NN$, satisfy the same. 

However, a more general question remains open as to whether restrictions of $\mg[\csub(L_1)]$ to finite-dimensional subspaces are continuous~\cite[Conjecture~1.3]{katsumasa2025magnitudegenericallycontinuousfinite}. One of the purposes of this paper is to contribute to the answer.

\section{Tractable Spaces}\label{section:tractable_spaces}

In this section, we introduce ``tractable metric spaces'', a class of metric spaces that appears convenient for the study of continuity of magnitude, and examine their properties.

First, let us recall a weaker notion. In every metric space, every compact subspace is closed and bounded. Recall that a metric space is called \df{Heine-Borel} (or \df{proper}) when the converse also holds, i.e.\ when its every closed bounded subspace is compact. This is equivalent to every closed ball being compact~\cite[1 Definitions, Chapter H]{Petrunin2023}.

\begin{lemma}\label{lemma:heine-borel-spaces}
The following statements hold for any Heine-Borel metric space~$M$.
\begin{enumerate}
\item
The space $M$ is complete (hence, a completion of its any dense subspace).
\item
Given any dense subspace $D \subseteq M$, a map $D \to \RR$ has a (necessarily unique) continuous extension $M \to \RR$ if and only if its restriction to every bounded subset of~$D$ is uniformly continuous.
\item
For every $K \in \icsub(M)$ and $r \in \RR_{\geq 0}$, the thickening 
\[\cth[M](K, r) = \set{x \in M}{d(x, K) \leq r}\]
is compact.
\item
The space $\icsub(M)$, equipped with the Hausdorff metric, is Heine-Borel.
\end{enumerate}
\end{lemma}

\begin{proof}
\
\begin{enumerate}
\item
Any Cauchy sequence is bounded and therefore contained in some compact ball. Thus, it has an accumulation point which for a Cauchy sequence is necessarily its limit.
\item
Uniqueness of a possible continuous extension follows from the assumption that $D$ is dense.

The statement is obvious for $M = \emptyset$, so assume that we have some $a \in M$. Suppose $D \subseteq X$ is dense and $f\colon D \to \RR$ such that it has a continuous extension $g\colon M \to \RR$. If $B \subseteq D$ is bounded, we have some $r \in \RR_{\geq 0}$ such that $B \subseteq \cball[M]{a}{r}$. The restriction $\rstr{g}_{\cball[M]{a}{r}}$ is continuous, therefore also uniformly continuous since $\cball[M]{a}{r}$ is compact. Hence, the further restriction $\rstr{f}_{B} = \rstr{g}_{B}$ is also uniformly continuous.

Conversely, suppose that all restrictions of $f\colon D \to \RR$ to bounded subsets of~$D$ are uniformly continuous. In particular, for every $x \in D$, the restriction $\rstr{f}_{\cball[D]{x}{1}}$ is uniformly continuous, so it has a unique continuous extension $\tilde{g}_x\colon \cl[M](\cball[D]{x}{1}) \to \RR$. Since $\oball[M]{x}{1} \subseteq \cl[M](\cball[D]{x}{1})$, we can further define $g_x\colon \oball[M]{x}{1} \to \RR$ as $g_x := \rstr{\tilde{g}_x}_{\oball[M]{x}{1}}$. Note that $\set{\oball[M]{x}{1}}{x \in D}$ is an open cover of~$M$, and the continuous maps $g_x$ match on the intersection of their domains by uniqueness of continuous extensions, hence they determine a common continuous extension $g\colon M \to \RR$ which is in particular a continuous extension of~$f$.
\item
Take some $a \in K$; then $K \subseteq \cball{a}{\mathrm{diam}(K)}$ and $\cth[M](K) \subseteq \cball{a}{\mathrm{diam}(K) + r}$. By assumption $\cball{a}{\mathrm{diam}(K) + r}$ is compact, and then so is its closed subspace $\cth[M](K, r)$.
\item
Take any $K \in \icsub(M)$ and $r \in \RR_{> 0}$. Then for every $L \in \cball[\icsub(M)]{K}{r}$ we have $L \subseteq \cth[M](K, r)$. Since $\cth[M](K, r)$ is compact by the previous item, the space $\icsub\big(\cth[M](K, r)\big)$ is compact by~\cite[Theorem 7.3.8]{burago2001course}. Hence, so is its closed subspace $\cball[\icsub(M)]{K}{r}$.
\end{enumerate}
\end{proof}

We can now state the definition of a tractable metric space.

\begin{definition}
A metric space $M$ is called \df{tractable} when it satisfies the following:
\begin{itemize}
\item
$M$ is positive definite (hence, $\mg[\fsub(M)]\colon \fsub(M) \to \RR$ is total and inclusion-monotone),
\item
all its closed balls are compact (i.e.~$M$ is Heine-Borel) and have finite magnitude:
\[\all{x \in M}\all{r \in \RR_{> 0}}{\cball{x}{r} \in \csub(M) \land \mg[\csub(M)]\big(\cball{x}{r}\big) < \infty}\]
(hence, magnitude $\mg[\csub(M)]\colon \csub(M) \to \RR$ is total).
\end{itemize}
\end{definition}

There is a straightforward characterization for when a normed vector space is tractable.

\begin{proposition}\label{proposition:tractable-normed-spaces}
The following statements are equivalent for any real normed vector space~$V$.
\begin{enumerate}[label=(\roman*)]
\item
$V$ is tractable.
\item
$V$ is finite dimensional and can be isometrically embedded into~$L_1$.
\item
$V$ is isometrically isomorphic to a finite-dimensional vector subspace of~$L_1$.
\item
$V$ is finite dimensional, positive definite, and all closed balls in it have finite magnitude.
\item
$V$ is finite dimensional, positive definite, and there exists a closed ball (with positive radius) in it with finite magnitude.
\end{enumerate}
\end{proposition}

\begin{proof}
Throughout the proof, we use the fact that a real normed vector space is Heine-Borel if and only if it is finite dimensional if and only if it is locally compact~\cite[Theorems 1.22, 1.23]{rudin1991functional} (and if a metric space is locally compact, so is its isometric copy).
\begin{itemize}
\item\proven{$(i) \implies (ii)$}
By assumption $V$ is Heine-Borel, so finite-dimensional.

We can write $V = \bigcup_{n \in \NN} \cball[V]{0}{n}$. Every compact metric space is separable, so as a countable union of such spaces, $V$ is separable as well. Hence, by~\cite[Corollary~3.5]{meckes2013}, $V$ isometrically embeds into~$L_1$.
\item\proven{$(ii) \implies (iii)$}
Since $V$ is finite-dimensional, it is locally compact and separable. Since it isometrically embeds into~$L_1$, it is positive definite. Hence, by~\cite[Corollary~3.5]{meckes2013}, there is a vector subspace $V' \subseteq L_1$ which is an isometric copy of~$V$. In particular, $V'$ is locally compact, so finite dimensional.
\item\proven{$(iii) \implies (iv)$}
Since $V$ is isometrically isomorphic to a finite-dimensional, therefore locally compact, vector space, it is itself finite-dimensional. Since $V$ is isometrically isomorphic to a subspace of~$L_1$, it is positive definite. Closed balls have finite magnitude by~\cite[Theorem~3.1]{leinster2023spaces}.
\item\proven{$(iv) \implies (i)$}
Every finite-dimensional real normed vector space is Heine-Borel.
\item\proven{$(iv) \implies (v)$}
Trivial.
\item\proven{$(v) \implies (iv)$}
Since translations are isometric isomorphisms, if we have $x \in V$, $r \in \RR_{> 0}$ such that $\mg[\csub(V)]\big(\cball[V]{x}{r}\big) < \infty$, then also $\mg[\csub(V)]\big(\cball[V]{0}{r}\big) < \infty$. It follows from~\cite[Corollary~2.4]{leinster2023spaces} that for every $R \in \RR_{> 0}$, the magnitude of~$\cball[V]{0}{R}$ is also finite. Every closed ball is contained in some~$\cball[V]{0}{R}$, so has finite magnitude by inclusion-monotonicity.
\end{itemize}
\end{proof}

The following proposition discusses tractability of subspaces.

\begin{proposition}\label{proposition:tractable-subspace}
The following holds for any metric space~$M$ and a subspace $A \subseteq M$.
\begin{enumerate}
\item
If $A$ is tractable, then it is closed in~$M$.
\item
If $M$ is tractable, then $A$ is tractable if and only if it is closed in~$M$.
\end{enumerate}
\end{proposition}

\begin{proof}
\
\begin{enumerate}
\item
Assume that $A$ is tractable, and take any $x \in \cl[M](A)$. Then there exists some $a \in A \cap \oball[M]{x}{1}$, and we have $x \in \cl[M](\cball[A]{a}{1})$. By assumption $\cball[A]{a}{1}$ is compact, therefore also closed in~$M$. Hence $x \in \cball[A]{a}{1} \subseteq A$. 
\item
The left-to-right implication holds by the previous item. Conversely, suppose that $A$ is closed in~$M$. Since $M$ is tractable, it is positive definite, and then so is the subspace~$A$. For any $a \in A$ and $r \in \RR_{> 0}$ we have $\cball[A]{a}{r} = \cball[M]{a}{r} \cap A$. The intersection of a compact and a closed subspace is compact. Also, by inclusion-monotonicity, since $\cball[M]{a}{r}$ has finite magnitude, so does $\cball[A]{a}{r}$.
\end{enumerate}
\end{proof}

We can now list standard examples of tractable spaces.

\begin{corollary}\label{corollary:tractable-examples}
\
\begin{enumerate}
\item
For every $p \in \intcc{1}{2}$ and every $N \in \NN$, the space~$\ell_p^N$ is tractable.
\item
Every finite-dimensional vector subspace of~$L_1$ is tractable.
\item
Every closed subspace of the above listed spaces is tractable.
\end{enumerate}
\end{corollary}

\begin{proof}
By Propositions~\ref{proposition:tractable-normed-spaces} and~\ref{proposition:tractable-subspace}.
\end{proof}

As stated, the spaces $\ell_p^N$ are tractable for $p \in \intcc{1}{2}$. Recall that for $p \in \intoo{0}{1}$, the usual definition of $p$-``norm'' is not in fact a norm, but its $p$-th power does induce a metric on~$\RR^N$, denoted by~$d_p$. These spaces also often appear in the context of magnitude, so we characterize their tractability here.

\begin{proposition}
The following statements are equivalent for every $N \in \NN$ and $p \in \intoo{0}{1}$.
\begin{enumerate}
\item
The space $(\RR^N, d_p)$ is tractable.
\item
The space $(\RR^N, d_p)$ is positive definite.
\item
There exists a linear map $T\colon \RR^N \to L_p(\intcc{0}{1})$ such that $\|T x\|_p = \|x\|$ for every $x \in \RR^N$.
\end{enumerate}
\end{proposition}

\begin{proof}
The equivalence between (ii) and~(iii) is a known result~\cite[Theorem~4.10]{LM17}. Statement (i) implies~(ii) by definition. It remains to verify that (ii) implies~(i).

Note that every closed ball in $(\RR^N, d_p)$ is also bounded and closed in, say, the $\infty$-metric, and the metrics $d_p$ and $d_\infty$ are equivalent. Hence all closed $p$-balls are compact. That they have finite magnitude follows from~\cite[Proposition~4.13]{LM17}.
\end{proof}

Next, we characterize compact tractable spaces.

\begin{proposition}\label{proposition:compact-tractable-spaces}
The following statements are equivalent for a metric space~$M$.
\begin{enumerate}[label=(\roman*)]
\item
$M$ is tractable and bounded.
\item
$M$ is tractable and compact.
\item
$M$ is positive definite, compact, and has finite magnitude.
\end{enumerate}
\end{proposition}

\begin{proof}
\
\begin{itemize}
\item\proven{$(i) \iff (ii)$}
Since $M$ is Heine-Borel and closed in itself.
\item\proven{$(ii) \implies (iii)$}
Immediate if $M = \emptyset$. Otherwise, $M$ is equal to some closed ball, and we are done.
\item\proven{$(iii) \implies (ii)$}
Every closed subset of compact~$M$ is compact, and has finite magnitude by inclusion-monotonicity.
\end{itemize}
\end{proof}

We want to show that the tractability property is naturally connected with the question of continuity of magnitude. We present the following two observations as motivation. First, we claim that in a locally compact metric space (so that we have enough interesting compact subsets), if magnitude is (finite and) continuous at some point, that point is necessarily contained in some tractable neighbourhood. Second, we claim that in sufficiently nice tractable spaces, the one-point property is essentially continuity of magnitude at singletons.

\begin{proposition}
Let $M$ be a locally compact positive definite metric space. Suppose that magnitude $\mg[\icsub(M)]\colon \icsub(M) \parto \RR$ is (defined and) continuous at some $K \in \icsub(M)$. Then there exists a tractable neighborhood $T \subseteq M$ of~$K$ (in particular, $\mg[\icsub(T)] = \rstr{\mg[\icsub(M)]}_{\icsub(T)}$ is a total map).
\end{proposition}

\begin{proof}
Let $K \in \icsub(M)$ have finite magnitude, and let $\mg[\icsub(M)]$ be continuous at~$K$. Since $M$ is locally compact, for every $x \in K$ we have some $r_x \in \RR_{> 0}$ such that the closed ball $\cball[M]{x}{r_x}$ is compact. The open balls $\set{\oball[M]{x}{r_x}}{x \in K}$ form an open cover of~$K$, so there exist $x_1, \ldots, x_n \in K$ such that $K \subseteq \oball[M]{x_1}{r_{x_1}} \cup \ldots \cup \oball[M]{x_n}{r_{x_n}} =: U$. If $U = M$, take $r := 1$, otherwise $r := d(K, M \setminus U)$. Either way $r \in \RR_{> 0}$ since the distance between a compact and a closed subset is positive if they are disjoint. Then, for every $\delta \in \intoo{0}{r}$, we have $\cth[M](K, \delta) \subseteq \cball[M]{x_1}{r_{x_1}} \cup \ldots \cup \cball[M]{x_n}{r_{x_n}}$. Since a finite union of compact spaces is compact, and $\cth[M](K, \delta)$ is closed, it is also compact. By the continuity of magnitude at~$K$, there actually exists $\delta \in \intoo{0}{r}$ such that $\mg[\icsub(M)]\big(\cth[M](K, \delta)\big) \leq \mg[\icsub(M)](K) + 1 < \infty$. Hence $T := \cth[M](K, \delta)$ is tractable by Proposition~\ref{proposition:compact-tractable-spaces}.
\end{proof}

\begin{proposition}\label{proposition:one-point-property-is-singleton-magnitude-continuity}
Let $V$ be a real normed vector space and let $S \subseteq V$ be tractable and star-shaped at $s \in S$. Then the following statements are equivalent.
\begin{enumerate}[label=(\roman*)]
\item
Every $K \in \icsub(S)$ has the one-point property.
\item
Magnitude $\mg[\icsub(S)]\colon \icsub(S) \to \RR$ is continuous at~$\set{s}$.
\end{enumerate}
Similarly, the following statements are equivalent.
\begin{enumerate}[label=(\roman*)]
\setcounter{enumi}{2}
\item
Every $F \in \ifsub(S)$ has the one-point property.
\item
Magnitude $\mg[\ifsub(S)]\colon \ifsub(S) \to \RR$ is continuous at~$\set{s}$.
\end{enumerate}
\end{proposition}

\begin{proof}
Since translations are isometric isomorphisms, we may without loss of generality assume $s = 0$.
\begin{itemize}
\item\proven{$(i) \implies (ii)$}
Take any $\epsilon \in \RR_{> 0}$. By the one-point property there exists $\delta \in \intoc{0}{1}$ such that
\[\mg[\icsub(S)]\big(\delta \cdot \cball[S]{0}{1}\big) \leq \mg[\icsub(S)]\big(\cball[S]{0}{\delta}\big) \leq 1 + \epsilon\]
(by inclusion-monotonicity of magnitude and the fact that $\delta \cdot \cball[S]{0}{1} \subseteq \cball[S]{0}{\delta}$ because $S$ is star-shaped at~$0$). Suppose $K \in \icsub(S)$ (so there is some $a \in K$) and $d_H(\set{0}, K) \leq \delta$. Then $\set{a} \subseteq K \subseteq \cball[S]{0}{\delta}$, so $1 \leq \mg[\icsub(S)](K) \leq 1 + \epsilon$. Hence $\big|\mg[\icsub(S)](\set{0}) - \mg[\icsub(S)](K)\big| = \big|1 - \mg[\icsub(S)](K)\big| \leq \epsilon$.
\item\proven{$(ii) \implies (i)$}
Take any $K \in \icsub(V)$ and $\epsilon \in \RR_{> 0}$. We have some $R \in \RR_{> 0}$ such that $K \subseteq \cball[S]{0}{R}$. By assumption, there exists $\delta \in \RR_{> 0}$ such that for every $L \in \icsub(S)$, if $d_H(\set{0}, L) \leq \delta$, then $\big|\mg[\icsub(S)](\set{0}) - \mg[\icsub(S)](L)\big| \leq \epsilon$. Hence, for $t \in \intoc{0}{\min\set{\frac{\delta}{R}, 1}}$, we have ($t \cdot K \subseteq S$ because $S$ is star-shaped at~$0$, and) $t \cdot K \subseteq \cball[S]{0}{\delta}$, so $d_H(\set{0}, t \cdot K) \leq \delta$, therefore $\mg[\icsub(S)](t \cdot K) \leq \mg[\icsub(S)](\set{0}) + \epsilon = 1 + \epsilon$.
\end{itemize}
The equivalence $(iii) \iff (iv)$ is proved in the same way.
\end{proof}

\begin{corollary}\label{corollary:one-point-property-is-singleton-magnitude-continuity}
Let $V$ be a real normed vector space and let $C \subseteq V$ be tractable and convex. Then the following statements are equivalent.
\begin{enumerate}[label=(\roman*)]
\item
Every $K \in \icsub(C)$ has the one-point property.
\item
For every $x \in C$, $\mg[\icsub(C)]\colon \icsub(C) \to \RR$ is continuous at~$\set{x}$.
\item
$C = \emptyset$ or there exists $x \in C$ such that $\mg[\icsub(C)]\colon \icsub(C) \to \RR$ is continuous at~$\set{x}$.
\end{enumerate}
Likewise for finite subsets instead of compact ones.
\end{corollary}

\begin{proof}
The proof quickly follows from Proposition~\ref{proposition:one-point-property-is-singleton-magnitude-continuity} and the fact that a subset is convex if and only if it is star-shaped at its every point.
\end{proof}

\begin{corollary}
If $V$ is a tractable real normed vector space, then $\mg[\icsub(V)]\colon \icsub(V) \to \RR$ is continuous at~$\set{x}$ for every $x \in V$.
\end{corollary}

\begin{proof}
By Proposition~\ref{proposition:tractable-normed-spaces} we may without loss of generality assume that $V$ is a finite-dimensional subspace of~$L_1$. By~\cite[Theorem~3.1]{leinster2023spaces}, all non-empty compact subspaces of~$V$ have the one-point property. The statement now follows from Corollary~\ref{corollary:one-point-property-is-singleton-magnitude-continuity}.
\end{proof}

While Proposition~\ref{proposition:tractable-normed-spaces} tells us that tractable vector subspaces of~$L_1$ are precisely the finite-dimensional ones, general tractable subspaces of~$L_1$ are more complicated; they are not just closed (cf.~Proposition~\ref{proposition:tractable-subspace}) subspaces of finite-dimensional vector subspaces.

For example, recall that $\ell_1$ isometrically embeds into~$L_1$. For $t \in \RR_{> 0}$ let $K_t$ be the closed convex hull of the set $\set{0} \cup \set{\tfrac{t}{2^n} e_n}{n \in \NN_{\geq 1}}$ in~$\ell_1$. By discussion in~\cite[Section~2]{leinster2023spaces}, $K_t$ is compact and has finite magnitude; specifically, by~\cite[Remark~2.2]{leinster2023spaces}, we have $\mg[\icsub(\ell_1)](K_t) \leq e^{t/2}$. Consider~$K_1$; by Proposition~\ref{proposition:compact-tractable-spaces}, it is tractable, but is clearly not contained in a finite-dimensional vector subspace. Since $t \cdot K_1 = K_t$ and $\lim_{t \to 0} e^{t/2} = 1$, we see that nevertheless $K_1$ has the one-point property.

This raises the question of whether every non-empty compact subspace of a tractable subspace of~$L_1$ (equivalently, by Proposition~\ref{proposition:tractable-subspace}: every compact tractable subspace of~$L_1$) has the one-point property. We do not yet know the answer, and the question is related to Conjecture~\ref{conjecture:tractable-magnitude-continuity} below. We do note, however, that the same idea as in the proof of~\cite[Theorem~3.1]{leinster2023spaces} does not work, as there it was crucial that the compact subspace was contained in a polytope (i.e.~convex hull of finitely many points, and therefore necessarily contained in a finite-dimensional vector subspace). See, however, Corollary~1.4 in~\cite{Meckes_2023} for a partial result (cf.~also Conjecture~1.5 in the same paper).

To wrap up this section, we note that in practice, it is not equally difficult to (dis)prove (uniform) continuity of magnitude for finite subspaces and compact subspaces. Specifically, it is usually easier to prove continuity for finite subspaces, and easier to disprove continuity for compact subspaces. For tractable spaces, the following proposition gives us the connection between the two, potentially allowing us to derive the more difficult case from the easier one. We will use this in the next section.

\begin{proposition}\label{proposition:compact-and-finite-magnitude-continuity}
Let $M$ be a tractable metric space.
\begin{enumerate}
\item
$\icsub(M)$ is a completion of~$\ifsub(M)$.
\item\label{proposition:compact-and-finite-magnitude-continuity:uniform-continuity}
 The map $\mg[\icsub(C)]\colon\icsub(M) \to \RR$ is uniformly continuous (resp.~$C$-Lipschitz) if and only if the map $\mg[\ifsub(M)]\colon\ifsub(M) \to \RR$ is uniformly continuous (resp.~$C$-Lipschitz).
\item\label{proposition:compact-and-finite-magnitude-continuity:compact-continuity}
The following statements are equivalent.
\begin{enumerate}
\item\label{proposition:compact-and-finite-magnitude-continuity:compact-continuity:magnitude-continuous}
The map $\mg[\icsub(C)]\colon\icsub(M) \to \RR$ is continuous.
\item\label{proposition:compact-and-finite-magnitude-continuity:compact-continuity:bounded-uniformly-continuous}
For every bounded subspace $B \subseteq M$, the map $\mg[\ifsub(B)]\colon\ifsub(B) \to \RR$ is uniformly continuous.
\item\label{proposition:compact-and-finite-magnitude-continuity:compact-continuity:ball-uniformly-continuous}
For every $x \in M$, $r \in \RR_{> 0}$, the map $\mg[\ifsub({\cball[M]{x}{r}})]\colon\ifsub\big(\cball[M]{x}{r}\big) \to \RR$ is uniformly continuous.
\end{enumerate}
If any (and thus all) of these equivalent statements hold, then $\mg[\ifsub(M)]\colon\ifsub(M) \to \RR$ is continuous.
\end{enumerate}
\end{proposition}

\begin{proof}
\
\begin{enumerate}
\item
By Lemma~\ref{lemma:heine-borel-spaces}.
\item
The map $\mg[\ifsub(M)]\colon \ifsub(M) \to \RR$ is a restriction of $\mg[\icsub(M)]\colon\icsub(M) \to \RR$, so if the latter is uniformly continuous/$C$-Lipschitz, so is the former. (Same for continuity, which proves the last part of item~3.)

Conversely, suppose that $\mg[\ifsub(M)]\colon \ifsub(M) \to \RR$ is uniformly continuous/$C$-Lipschitz. Then it has a unique continuous extension $\icsub(M) \to \RR$ which is also uniformly continuous/$C$-Lipschitz. But for positive definite spaces this extension has to be the magnitude $\mg[\icsub(M)]\colon\icsub(M) \to \RR$~\cite[Corollary 2.7]{meckes2013}.
\item
\begin{itemize}
\item\proven{$\ref{proposition:compact-and-finite-magnitude-continuity:compact-continuity:magnitude-continuous} \implies \ref{proposition:compact-and-finite-magnitude-continuity:compact-continuity:bounded-uniformly-continuous}$}
If $B \subseteq M$ is bounded, then so is $\ifsub(B)$ (in the Hausdorff metric). The claim then follows from Lemma~\ref{lemma:heine-borel-spaces}.
\item\proven{$\ref{proposition:compact-and-finite-magnitude-continuity:compact-continuity:bounded-uniformly-continuous} \implies \ref{proposition:compact-and-finite-magnitude-continuity:compact-continuity:ball-uniformly-continuous}$}
Balls are bounded.
\item\proven{$\ref{proposition:compact-and-finite-magnitude-continuity:compact-continuity:ball-uniformly-continuous} \implies \ref{proposition:compact-and-finite-magnitude-continuity:compact-continuity:magnitude-continuous}$}
Take any bounded $B \subseteq \ifsub(M)$. We want to prove that the restriction of the $\mg[\ifsub(M)]\colon \ifsub(M) \to \RR$ to~$B$ is uniformly continuous. This is clear if $B = \emptyset$. Otherwise, pick some $F \in B$ and $a \in F$. Let $R := 1 + \mathrm{diam}(F) + \mathrm{diam}(B)$; then $R \in \RR_{> 0}$ and $B \subseteq \ifsub\big(\oball[M]{a}{R}\big)$. The required uniform continuity now follows from the assumption. Hence, $\mg[\ifsub(M)]\colon \ifsub(M) \to \RR$ has a unique continuous extension to $\mg[\icsub(M)]\colon \icsub(M) \to \RR$ by Lemma~\ref{lemma:heine-borel-spaces}. By the same argument as in item~2, this continuous extension has to be the magnitude function for compact subspaces.
\end{itemize}
\end{enumerate}
\end{proof}

\section{Continuity of Magnitude on~\texorpdfstring{$\RR$}{R} and subspaces of \texorpdfstring{$\RR$}{R}}\label{section:continuity_of_magnitude}

The goal of this section is to deduce that magnitude $\mg[\icsub(\RR)]\colon \icsub(\RR) \to \RR$ is continuous from the theory developed in the previous section. This is difficult to prove directly, even though the result has appeared in the literature previously (for example, see~\cite[Corollary 5.3.13]{LM17} and \cite[Corollary 2.12]{meckes2013}). We also prove that the restrictions of $\mg[\icsub(\RR)]$ to bounded subspaces are Lipschitz.

The strategy is to derive these results from the continuity of magnitude $\mg[\ifsub(\RR)]\colon \ifsub(\RR) \to \RR$, using Proposition~\ref{proposition:compact-and-finite-magnitude-continuity}. The simplest way to do this would be to use the second item in this proposition, but it turns out that magnitude is not uniformly continuous on the whole~$\RR$, which we now show. Instead, we will use the third item in the proposition, i.e.\ we will show that magnitude is uniformly continuous on bounded subsets of~$\RR$.

\begin{proposition}\label{proposition:magnitude-not-uniformly-continuous}
The space $\RR_{\geq 0}$ is tractable, but the magnitude maps 
\[
\mg[\ifsub(\RR_{\geq 0})]\colon \ifsub(\RR_{\geq 0}) \to \RR
\]
and 
\[\mg[\icsub(\RR_{\geq 0})]\colon \icsub(\RR_{\geq 0}) \to \RR
\]
are not uniformly continuous (therefore also not Lipschitz continuous). Hence, the same is also true for every tractable metric space containing an isometric copy of~$\RR_{\geq 0}$.
\end{proposition}

\begin{proof}
Tractability of~$\RR_{\geq 0}$ follows immediately from Propositions~\ref{corollary:tractable-examples} and~\ref{proposition:tractable-subspace}.

Uniform continuity of the magnitude $\mg[\ifsub(\RR_{\geq 0})]\colon \ifsub(\RR_{\geq 0}) \to \RR$ is equivalent to that of the magnitude $\mg[\icsub(\RR_{\geq 0})]\colon\icsub(\RR_{\geq 0}) \to \RR$ by Proposition~\ref{proposition:compact-and-finite-magnitude-continuity}(\ref{proposition:compact-and-finite-magnitude-continuity:uniform-continuity}). Let us prove that the magnitude $\mg[\icsub(\RR_{\geq 0})]\colon\icsub(\RR_{\geq 0}) \to \RR$ is not uniformly continuous. In fact, we show that uniform continuity fails for every $\epsilon$.

Take any $\epsilon, \delta \in \RR_{> 0}$; in particular that means $\delta > \tanh(\delta)$. Pick $n \in \NN_{\geq 2}$ such that $n \geq \frac{\epsilon}{\delta - \tanh(\delta)}$. Define $A := \set[1]{(2k + 1) \delta}{k \in \NN_{< n}}$ and $B := \intcc{0}{2n \delta}$. Clearly, $d_H(A, B) = \delta$. 
Since $A \subseteq B$, we have $\mg[\icsub(\RR_{\geq 0})](A) \leq \mg[\icsub(\RR_{\geq 0})](B)$, and
\[
\begin{array}{rcl}
\mg[\icsub(\RR_{\geq 0})](B) - \mg[\icsub(\RR_{\geq 0})](A) &=& (1 + n \delta) - \big(1 + (n-1) \tanh(\delta)\big) \\&=& n \cdot \big(\delta - \tanh(\delta)\big) + \tanh(\delta)\\
&>& \epsilon.
\end{array}\]
\end{proof}

\begin{remark}\label{remark:magnitude-continuity-on-natural-numbers}
In spite of Proposition~\ref{proposition:magnitude-not-uniformly-continuous}, it is still possible for magnitude to be uniformly continuous on an unbounded tractable space --- $\NN$ is a case in point. This is because the Hausdorff distance between any two distinct compact subsets of~$\NN$ is at least~$1$, so we can take $\delta = \tfrac{1}{2}$.

However, magnitude is still not Lipschitz on~$\NN$. Take any $C \in \RR_{> 0}$. Since $2\tanh(\tfrac{1}{2}) - \tanh(1) \approx 0.16$ is positive, we may find such $N \in \NN_{\geq 1}$ that $N > \frac{C}{2\tanh(\frac{1}{2}) - \tanh(1)}$. Let $A := \set{2k}{k \in \NN_{\leq N}}$ and $B := \NN_{\leq 2N}$. Then $d_H(A, B) = 1$ and
\[\big|\mg[\icsub(\NN)](A) - \mg[\icsub(\NN)](B)\big| = \mg[\icsub(\NN)](B) - \mg[\icsub(\NN)](A) =\]
\[= \big(1 + 2N \tanh(\tfrac{1}{2})\big) - \big(1 + N \tanh(1)\big) = N \cdot \big(2\tanh(\tfrac{1}{2}) - \tanh(1)\big) > C.\]
\end{remark}

We now focus on bounded intervals $\intcc{a}{b} \subseteq \RR$ and prove that their magnitude maps are Lipschitz continuous with an explicit constant depending on the size of the interval.

\begin{theorem}\label{theorem:lipschitz-continuity-of-magnitude-on-intervals}
For all $a, b \in \RR$, $a < b$, the magnitude maps $\mg[\ifsub(\intcc{a}{b})]\colon \ifsub(\intcc{a}{b}) \to \RR$ and $\mg[\icsub(\intcc{a}{b})]\colon \icsub(\intcc{a}{b}) \to \RR$ are Lipschitz with coefficient $1 + \tfrac{b - a}{2}$.
\end{theorem}

To prove this theorem, we will need the following lemmas.

\begin{lemma}\label{lemma:tanh-inequality}
For all $x \in \RR_{> 0}$ we have $\frac{\tanh(x)}{x} \geq 1 - \tanh^2(x)$.
\end{lemma}

\begin{proof}
Since $\sinh(0) = 0$ and $\sinh'(x) = \cosh(x) \geq 1$, we have $\sinh(x) \geq x$ for all $x \in \RR_{> 0}$. Then
\[1 \leq \frac{\sinh(2x)}{2x} = \frac{\sinh(x) \cosh(x)}{x} = \frac{\tanh(x) \cosh^2(x)}{x}.\]
Hence
\[\frac{\tanh(x)}{x} \geq \frac{1}{\cosh^2(x)} = \frac{\cosh^2(x) - \sinh^2(x)}{\cosh^2(x)} = 1 - \tanh^2(x).\]
\end{proof}

\begin{lemma}\label{lemma:growth-of-magnitude-of-thickening}
For all $A \in \ifsub(\RR)$ and $r \in \RR_{\geq 0}$ we have
\[\mg[\icsub(\RR)]\big(\cth{A}{r}\big) - \mg[\icsub(\RR)](A) \leq r \cdot \big(1 + \tfrac{\max{A} - \min{A}}{2}\big).\]
\end{lemma}

\begin{proof}
The statement is obviously true for $r = 0$, so assume $r > 0$. Let $n$ denote the cardinality of~$A$ and write $A = \set{a_1, \ldots, a_n}$ where $a_1 < \ldots < a_n$. Observe that for any $k \in \intco[\NN]{1}{n}$, if $a_{k+1}$ is added, the increase in magnitude of the thickening, i.e.
\[\mg[\icsub(\RR)]\big(\cth{\set{a_1, \ldots, a_k, a_{k+1}}}{r}\big) - \mg[\icsub(\RR)]\big(\cth{\set{a_1, \ldots, a_k}}{r}\big),\]
is equal to $r + \tanh(\tfrac{a_{k+1} - a_k}{2} - r)$ if $r \leq \tfrac{a_{k+1} - a_k}{2}$, and to $\tfrac{a_{k+1} - a_k}{2}$ if $r \geq \tfrac{a_{k+1} - a_k}{2}$. Let $I := \set[1]{i \in \intco[\NN]{1}{n}}{r \leq \tfrac{a_{i+1} - a_i}{2}}$ and $J := \intco[\NN]{1}{n} \setminus I$. Then
\[\mg[\icsub(\RR)]\big(\cth{A}{r}\big) - \mg[\icsub(\RR)](A) \leq\]
\[\leq n \cdot r + \sum_{i \in I} \big(\tanh(\tfrac{a_{i+1} - a_i}{2} - r) - \tanh(\tfrac{a_{i+1} - a_i}{2})\big) + \sum_{j \in J} \big(\tfrac{a_{j+1} - a_j}{2} - \tanh(\tfrac{a_{j+1} - a_j}{2}) - r\big) =\]
\[= r \cdot \Big(n - \sum_{i \in I} \frac{\tanh(\tfrac{a_{i+1} - a_i}{2} - r) - \tanh(\tfrac{a_{i+1} - a_i}{2})}{-r} - \sum_{j \in J} \big(1 - \tfrac{1}{r} \big(\tfrac{a_{j+1} - a_j}{2} - \tanh(\tfrac{a_{j+1} - a_j}{2})\big)\big)\Big).\]
The term $\frac{\tanh(\tfrac{a_{i+1} - a_i}{2} - r) - \tanh(\tfrac{a_{i+1} - a_i}{2})}{-r}$ is the slope of the secant line which intersects the graph of~$\tanh(x)$ at $x = \tfrac{a_{i+1} - a_i}{2} - r$ and $x = \tfrac{a_{i+1} - a_i}{2}$. As $\tanh$ is concave on nonnegative inputs, this is bounded from below by the slope of the tangent to the graph of~$\tanh$ at $\tfrac{a_{i+1} - a_i}{2}$:
\[\frac{\tanh(\tfrac{a_{i+1} - a_i}{2} - r) - \tanh(\tfrac{a_{i+1} - a_i}{2})}{-r} \geq \tanh'(\tfrac{a_{i+1} - a_i}{2}) = 1 - \tanh^2(\tfrac{a_{i+1} - a_i}{2}).\]
For $j \in J$, denoting $u := \tfrac{a_{j+1} - a_j}{2}$ and using Lemma~\ref{lemma:tanh-inequality}, we also have
\[1 - \tfrac{1}{r} \big(u - \tanh(u)\big) \geq 1 - \tfrac{1}{u} \big(u - \tanh(u)\big) = \tfrac{\tanh(u)}{u} \geq 1 - \tanh^2(u).\]
Hence,
\[\frac{\mg[\icsub(\RR)]\big(\cth{A}{r}\big) - \mg[\icsub(\RR)](A)}{r} \leq n - \sum_{k \in \intco[\NN]{1}{n}} \big(1 - \tanh^2(\tfrac{a_{k+1} - a_k}{2})\big) =\]
\[= 1 + \sum_{k \in \intco[\NN]{1}{n}} \tanh^2(\tfrac{a_{k+1} - a_k}{2}) \leq 1 + \sum_{k \in \intco[\NN]{1}{n}} \tanh(\tfrac{a_{k+1} - a_k}{2}) \leq\]
\[\leq 1 + \sum_{k \in \intco[\NN]{1}{n}} \tfrac{a_{k+1} - a_k}{2} = 1 + \tfrac{\max{A} - \min{A}}{2}.\]
\end{proof}

\begin{proof}[Proof of Theorem~\ref{theorem:lipschitz-continuity-of-magnitude-on-intervals}]
By Proposition~\ref{proposition:compact-and-finite-magnitude-continuity}(\ref{proposition:compact-and-finite-magnitude-continuity:uniform-continuity}) it suffices to prove the Lipschitz property for  $\mg[\ifsub(\intcc{a}{b})]\colon \ifsub(\intcc{a}{b}) \to \RR$.

Take any $A, B \in \ifsub(\intcc{a}{b})$ and set $\delta := d_H(A, B)$. Then $B \subseteq \cth{A}{\delta}$, so by inclusion-monotonicity of magnitude and Lemma~\ref{lemma:growth-of-magnitude-of-thickening}
\[
\begin{array}{rcl}\mg[\ifsub(\intcc{a}{b})](B) - \mg[\ifsub(\intcc{a}{b})](A) &=&
\mg[\icsub(\intcc{a}{b})](B) - \mg[\icsub(\intcc{a}{b})](A)\\
&\leq& \mg[\icsub(\intcc{a}{b})]\big(\cth{A}{\delta}\big) - \mg[\icsub(\intcc{a}{b})](A)\\ 
&\leq & \delta \cdot \big(1 + \tfrac{\max{A} - \min{A}}{2}\big) \leq \delta \cdot \big(1 + \tfrac{b - a}{2}\big).
\end{array}\]
Exchanging the roles of $A$ and~$B$, we likewise get 
\[\mg[\ifsub(\intcc{a}{b})](A) - \mg[\ifsub(\intcc{a}{b})](B) \leq \delta \cdot \big(1 + \tfrac{b - a}{2}\big).\] Hence,
\[\big|\mg[\ifsub(\intcc{a}{b})](A) - \mg[\ifsub(\intcc{a}{b})](B)\big| \leq \big(1 + \tfrac{b - a}{2}\big) \cdot d_H(A, B)\]
which gives us the desired Lipschitz continuity.
\end{proof}


We can now conclude the continuity of magnitude on~$\RR$.

\begin{theorem}\label{thm:continuity_on_R}
The magnitude map $\mg[\icsub(\RR)]\colon \icsub(\RR) \to \RR$ is continuous, and its restrictions to bounded subspaces are Lipschitz.
\end{theorem}

\begin{proof}
Every bounded subset of~$\RR$ is contained in some bounded closed interval, so this follows from Theorem~\ref{theorem:lipschitz-continuity-of-magnitude-on-intervals} and Proposition~\ref{proposition:compact-and-finite-magnitude-continuity}(\ref{proposition:compact-and-finite-magnitude-continuity:compact-continuity}).
\end{proof}

\section{Conclusion}\label{section:contribution}

In summary, in this paper we introduce the class of \emph{tractable metric spaces} as a framework for studying the continuity of magnitude. Our key contributions include:

 \begin{enumerate}
  \item Introducing tractable spaces: positive definite, Heine--Borel metric spaces in which all closed balls have finite magnitude.
  \item Offering several equivalent characterizations of tractable normed vector spaces.
  \item Providing a characterization of when the one-point property holds in tractable star-shaped spaces and showing that it holds for all tractable normed vector spaces.
  \item Establishing that the continuity of magnitude on the space of compact subspaces of a tractable space $M$ is equivalent to uniform continuity on finite and compact subspaces of bounded subspaces $B \subseteq M$.
  \item Applying this framework to the real line: we provide a new proof that magnitude is continuous on compact subspaces of~\(\mathbb{R}\) (with the Hausdorff metric), and show that it is Lipschitz on bounded subspaces.
\end{enumerate}

Although we have shown here continuity of magnitude only for the subspaces of $\RR$, we suspect that it holds much more generally. We state this as the following conjecture, which is a generalization of~\cite[Conjecture~1.3]{katsumasa2025magnitudegenericallycontinuousfinite} (see also discussions in \cite{Meckes-volume}, \cite{Meckes_2023}).

\begin{conjecture}\label{conjecture:tractable-magnitude-continuity}
If the subspace $T \subseteq L_1$ is tractable, then $\mg[\ifsub(T)]\colon \ifsub(T) \to \RR$ and $\mg[\icsub(T)]\colon\icsub(T) \to \RR$ are continuous.
\end{conjecture}

Additionally, we state the following conjecture about Lipschitz continuity of magnitude (cf.~also Remark~\ref{remark:magnitude-continuity-on-natural-numbers}).

\begin{conjecture}
Let $T \subseteq L_1$ be tractable. Then $\mg[\icsub(T)]\colon \icsub(T) \to \RR$ is Lipschitz (equivalently, $\mg[\ifsub(T)]\colon \ifsub(T) \to \RR$ is Lipschitz, by Proposition~\ref{proposition:compact-and-finite-magnitude-continuity}(\ref{proposition:compact-and-finite-magnitude-continuity:uniform-continuity})) if and only if $T$ is bounded (equivalently, $T$ is compact, by Proposition~\ref{proposition:compact-tractable-spaces}).
\end{conjecture}

The two conjectures are related: by Proposition~\ref{proposition:compact-and-finite-magnitude-continuity}(\ref{proposition:compact-and-finite-magnitude-continuity:compact-continuity}), the latter implies the former.

\bibliographystyle{plain}
\bibliography{magnitude}


%
%
%
%
%
%

%
\end{document}%